\documentclass[a4paper,reqno,11pt]{amsart}
\usepackage[a4paper,margin=1in]{geometry}
\usepackage{graphicx}
\usepackage{amsmath,amssymb,amsthm}
\usepackage{hyperref}

\theoremstyle{plain}
\newtheorem{theorem}{Theorem}[section]

\newtheorem{proposition}[theorem]{Proposition}
\newtheorem{corollary}[theorem]{Corollary}
\theoremstyle{definition}

\numberwithin{equation}{section}

\newcommand{\R}{\mathbb{R}}
\newcommand{\abs}[2][]{#1\lvert #2 #1\rvert}

\DeclareMathOperator{\sign}{\mathrm sign}

\newcommand{\flow}{\Phi} 

\newcommand{\FF}{{\mathcal S}}

\title{Nonexistence of subcritical solitary waves}
\author{Vladimir Kozlov}
\address{Department of Mathematics, Linköping University, SE-581 83 Linköping, Sweden}
\author{Evgeniy Lokharu}
\address{Department of Mathematics, Linköping University, SE-581 83 Linköping, Sweden}
\author{Miles H.~Wheeler}
\address{Department of Mathematical Sciences, University of Bath, Bath, BA2 7AY, UK}
\date{\today}
\begin{document}
\begin{abstract}
  We prove the nonexistence of two-dimensional solitary gravity water waves with subcritical wave speeds and an arbitrary distribution of vorticity. This is a longstanding open problem, and even in the irrotational case there are only partial results relying on sign conditions or smallness assumptions. As a corollary, we obtain a relatively complete classification of solitary waves: they must be supercritical, symmetric, and monotonically decreasing on either side of a central crest. 
  The proof introduces a new function which is related to the so-called flow force and has several surprising properties. 
  In addition to solitary waves, our nonexistence result applies to ``half-solitary'' waves (e.g.~bores) which decay in only one direction.
\end{abstract}
\maketitle

\section{Introduction} \label{s:introduction}

Solitary water waves are localized disturbances of a fluid surface which travel at constant speed without change of form. They were first discovered by John Scott Russell in 1834 \cite{russell}, who in subsequent experiments observed the empirical relationship
\begin{equation}
  \label{eqn:empirical}
  F^2 = \frac{c^2}{gd} \approx 1 + \frac ad
\end{equation}
between the speed $c$ of the wave and its amplitude $a > 0$. Here $g$ is the acceleration due to gravity, $d$ is depth of the fluid far away from the wave, and $F$ is a dimensionless wave speed called the \emph{Froude number}. While the approximation \eqref{eqn:empirical} is only valid for waves with small amplitude, all of the solitary waves which have so far been rigorously constructed~\cite{Lavrentiev43,FriedrichsHyers54,Beale77,Mielke88,AmickToland81b} are nevertheless \emph{supercritical} in that they satisfy the strict inequality
\begin{equation}
  \label{eqn:supercritical}
  F > 1.
\end{equation}
Using the method of moving planes, Craig and Sternberg~\cite{CraigSternberg88} showed in 1988 that the free surface of a supercritical solitary wave is necessarily symmetric and strictly decreasing on either side of a central crest. They further conjectured that all solitary waves are supercritical, and hence subject to their result. In the present paper we positively resolve this conjecture. Previous work was restricted to \emph{waves of elevation} whose free surface lies everywhere above its asymptotic level~\cite{starr,AmickToland81b,mcleod}, to symmetric and monotone \emph{waves of depression} where the reverse inequality holds~\cite{kp}, or to nearly critical waves with $F \approx 1$~\cite{kk:nearcrit}.

\subsection{Historical discussion}

Russell's observations of solitary waves, including the formula \eqref{eqn:empirical}, are famously inconsistent with Airy's linear shallow water theory. This discrepancy was finally explained using weakly nonlinear models by Boussinesq in 1871 and Rayleigh in 1876. Perhaps the most famous such model is the so-called KdV equation, first written down by Boussinesq in 1877 but later rediscovered by Korteweg and de~Vries in 1895; see the discussion in \cite{miles:survey}. Traveling wave solutions of the KdV equation satisfy a second order ordinary differential equation. For subcritical Froude numbers $F<1$, the equilibrium at the origin is a center, which immediately rules out the existence of solitary waves (i.e.~homoclinic orbits).

In higher-order models, however, the traveling wave ODE has a saddle-center at the origin when $F<1$, and the existence of solitary wave solutions is quite subtle. Generically, one expects so-called ``generalized solitary waves'' that are homoclinic to a small periodic solution, while true solitary waves, referred to as ``embedded solitary waves'' in this context, are now interpreted as a codimension-one phenomenon where the amplitude of this periodic solution vanishes~\cite{Champneys01}. For a general introduction to this topic we refer the reader to the monographs by Boyd~\cite{Boyd98} and Lombardi~\cite{Lombardi00} and the references therein. 

\begin{figure}
  \centering
  \includegraphics[scale=1]{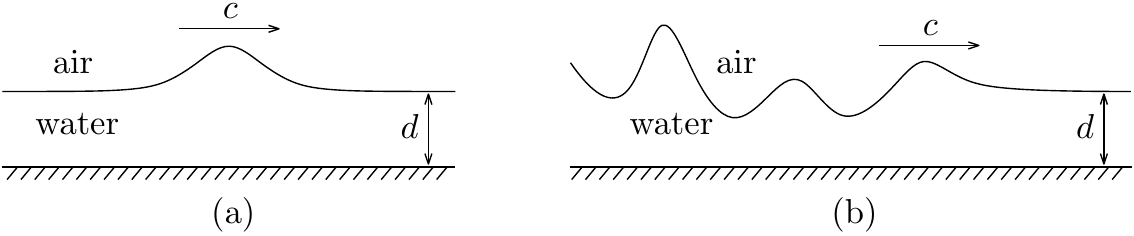}
  \caption{(a) A solitary wave which is symmetric and monotone about a central crest. (b) A ``half-solitary wave'' which converges to some limit as $x\to\infty$ but not necessarily as $x \to -\infty$.}
  \label{fig:halfsolitary}
\end{figure}

The rigorous existence of small-amplitude solitary wave solutions to the full water wave problem goes back to Lavrentiev~\cite{Lavrentiev43} in 1943 and Friedrichs and Hyers~\cite{FriedrichsHyers54} in 1954. Subsequent proofs used more modern methods; Beale~\cite{Beale77} in 1977 used the Nash--Moser implicit function theorem while Mielke~\cite{Mielke88} in 1988 used center-manifold techniques for infinite-dimensional dynamical systems. Note that, for instance by \eqref{eqn:empirical}, the inequality $F > 1$ defining supercriticality is sharp in the small-amplitude limit. Waves with large or ``finite'' amplitude were first constructed by Amick and Toland~\cite{AmickToland81b} in 1981 using global bifurcation techniques, leading to the existence of a limiting extreme wave with an angled crest~\cite{AmickFraenkelToland82}; also see~\cite{AmickToland81a,BenjaminBonaBose90}.
All of these waves are supercritical waves of elevation which are symmetric and \emph{monotone} in that their free surfaces are strictly decreasing on either side of a central crest as in Figure~\ref{fig:halfsolitary}(a). Indeed, especially for large amplitudes, establishing these properties is a key step in the construction.

In 1947, Starr~\cite{starr} discovered a surprisingly simple identity relating the wave speed to the integral of the free surface elevation (the ``excess mass'') and the integral of its square (related to the potential energy). For waves of elevation, this identity immediately implies supercriticality. Amick and Toland provided a rigorous proof of Starr's formula~\cite{AmickToland81b}, and a much simpler argument was given by McLeod~\cite{mcleod}. Keady and Pritchard~\cite{kp} showed that there are no monotone and symmetric solitary waves of depression, while Kozlov and Kuznetsov~\cite{kk:nearcrit} ruled out the existence of mildly subcritical solitary waves with $F < 1$ and $F \approx 1$.

In the infinite-depth limit $d=\infty$, \eqref{eqn:supercritical} can never hold so that in some sense all waves are subcritical and one expects nonexistence. Indeed, S.~M.~Sun~\cite{sun:nonexistence} proved integral identities which rule out the existence of solitary waves of elevation or depression under an assumption on the decay rate of the free surface. Hur~\cite{hur:nosolitary} removed the sign condition by proving a different identity in conformal variables, and recently Ifrim and Tataru~\cite{it:nosolitary} substantially weakened the decay assumption. For waves of elevation or depression, Craig~\cite{craig:nonexistence} gave an elegant nonexistence proof using the maximum principle. This argument has no assumptions whatsoever on the decay rate, and can be generalized to finite-depth waves in three dimensions.

The above results consider the classical water wave problem where the fluid velocity field is assumed to be irrotational. Recently, however, there has been considerable mathematical interest in steady water waves with vorticity; see for instance the surveys~\cite{groves:survey, strauss:survey}. Here we will confine our discussion to the \emph{unidirectional} case where the horizontal fluid velocity in the moving frame has a definite sign. As for irrotational waves, one can define a dimensionless wave speed $F$ so that the critical threshold is $F = 1$ (see \eqref{eqn:F} below), and we will continue to call this quantity the Froude number.
As in the irrotational case, all of the small-amplitude~\cite{ter:rot,hur:exact,GrovesWahlen08} and large-amplitude~\cite{Wheeler13} solitary waves with vorticity that have been constructed are supercritical and monotone waves of elevation. Moreover, supercritical waves are automatically waves of elevation~\cite{Wheeler13,Kozlov2015}, and hence monotone and symmetric~\cite{Hur2008}, and conversely all waves of elevation are supercritical~\cite{Wheeler15b}. Once again there are no mildly subcritical solitary waves~\cite{Kozlov2017a}.

While outside the scope of the present paper, there is also a large literature concerning solitary water waves with the additional effects of surface tension or density stratification. We refer the interested reader to the surveys~\cite{groves:survey,Helfrich2006} and in particular highlight the nonexistence result \cite{sun:nonexistence} and symmetry results in \cite{CraigSternberg92}.

\subsection{Statement of the main result}
Consider a two-dimensional fluid region 
\begin{align*}
  D = \{(x,y) \in \R^2 : 0 < y < d+\eta(x,t)\}
\end{align*}
bounded below by a flat bed $y=0$ and above by a \emph{free surface} $y=d+\eta(x,t)$. Here we think of $d$ as some average or reference depth for the fluid, and $\eta$ as the deviation of the free surface from this level. Writing $(u,v)$ for the components of the fluid velocity and $P$ for the ratio of the pressure and the constant density $\rho$, the incompressible Euler equations in $D$ are
\begin{subequations}\label{eqn:euler}
  \begin{align}
    \label{eqn:euler:u}
    u_t + uu_x + vu_y &= -P_x, \\
    \label{eqn:euler:v}
    v_t + uv_x + vv_y &= -P_y-g, 
  \end{align}
  together with the incompressibility constraint
  \begin{equation}
    \label{eqn:euler:incomp}
    u_x + v_y = 0,
  \end{equation}
  where $g > 0$ is the constant acceleration due to gravity. At the bed $y=0$ we impose the kinematic condition
  \begin{equation}\label{eqn:euler:kinbot}
    v = 0,
  \end{equation}
  while on the free surface $y=d+\eta(x,t)$ we impose both a kinematic
  condition
  \begin{equation}\label{eqn:euler:kintop}
    v = \eta_t + u\eta_x 
  \end{equation}
  and the dynamic boundary condition that
  \begin{equation}\label{eqn:euler:dynamic}
    P = P_{\mathrm{atm}} = \text{constant}.
  \end{equation}
  The kinematic boundary conditions assert that fluid particles on the boundary of the fluid region must remain there for all time, while the dynamic boundary expresses the balance of forces on the free surface.
\end{subequations}

By a \emph{traveling wave} we mean a solution of \eqref{eqn:euler} where $u,v,P,\eta$ depend on $x$ and $t$ only through the combination $x-ct$ for some constant wave speed $c$. In this case \eqref{eqn:euler:u}--\eqref{eqn:euler:incomp} are equivalent to the time-independent equations
\begin{subequations}\label{eqn:trav}
  \begin{align}
    \label{eqn:u}
    (u-c)u_x + vu_y & = -P_x,   \\
    \label{eqn:v}
    (u-c)v_x + vv_y & = -P_y-g, \\
    \label{eqn:incomp}
    u_x + v_y &= 0 
  \end{align}
  in the time-independent fluid domain $0 < y < d+\eta(x)$, while the boundary conditions \eqref{eqn:euler:kinbot}--\eqref{eqn:euler:dynamic} become
  \begin{alignat}{2}
    \label{eqn:kinbot}
    v &= 0&\qquad& \text{on } y=0,\\
    \label{eqn:kintop}
    v &= (u-c)\eta_x && \text{on } y=d+\eta,\\
    \label{eqn:dyn}
    P &= P_{\mathrm{atm}} && \text{on } y=d+\eta.
  \end{alignat}
\end{subequations}
In this paper we are interested in traveling waves which satisfy the
asymptotics
\begin{equation}
  \label{eqn:asym}
  \eta(x) \to 0,
  \quad
  u(x,y) \to U(y),
  \quad 
  v(x,y) \to 0
  \qquad 
  \text{as } x \to +\infty
\end{equation}
for some function $U(y)$. This uniquely defines $d$ as the asymptotic depth of the fluid, and also determines the asymptotic distribution of the vorticity 
\begin{align*}
  \omega=v_x-u_y.
\end{align*}
Indeed, if $\omega \equiv 0$ so that the flow is irrotational, $U(y)$ must be constant. When the limits in \eqref{eqn:asym} also hold as $x \to -\infty$, the wave is called \emph{solitary}. For want of a better term we therefore call solutions of \eqref{eqn:trav}--\eqref{eqn:asym} ``half-solitary waves''; see Figure~\ref{fig:halfsolitary}(b).

One of our central assumptions is that the wave is \emph{unidirectional} in the sense that the relative fluid velocity $u-c$ in the moving frame never vanishes. Without loss of generality we assume $u-c > 0$, and more precisely require 
\begin{equation}
  \inf (u-c) > 0  \label{a:unidirectional}.
\end{equation} 
In the irrotational case where $\omega\equiv 0$, a maximum principle argument implies that \eqref{a:unidirectional} is always satisfied~\cite{Toland96}.

In particular, \eqref{a:unidirectional} means that $U-c>0$, and so we can consider the dimensionless quantity $F > 0$ defined by
\begin{equation}
  \label{eqn:F}
  \frac 1{F^2} = 
  g\int_0^d \frac{dy}{(U(y)-c)^2},
\end{equation}
which we will call the \emph{Froude number}; see~\cite{Wheeler15b} for a detailed discussion of this terminology. In the irrotational case one traditionally works a reference frame where $U \equiv 0$, in which case \eqref{eqn:F} recovers the classical definition $F = c/\sqrt{gd}$. We call a solution to \eqref{eqn:trav} \emph{subcritical} if $F < 1$, \emph{critical} if $F=1$, and \emph{supercritical} if $F > 1$. The value $F = 1$ is critical in the sense that the linearized equations about the ``trivial'' solution
\begin{align}
  \label{eqn:triv}
  u(x,y)=U(y),
  \quad 
  v \equiv 0,
  \quad 
  \eta \equiv 0,
  \quad 
  P(x,y)=P_{\mathrm{atm}}+g(y-d),
\end{align}
have nontrivial bounded solutions if and only if $F \le 1$; if $F = 1$, the solution depends only on the vertical variable $y$.

Somewhat informally stated, our main result is the following. 

\begin{theorem}\label{t:informal}
  The only bounded classical solutions of \eqref{eqn:trav}--\eqref{a:unidirectional} which are subcritical are the trivial solutions \eqref{eqn:triv}.
\end{theorem}
See Theorem~\ref{t:main} for a precise version. Combining Theorem~\ref{t:informal} with the elevation and symmetry results in \cite{CraigSternberg88,Hur2008,Wheeler13,Kozlov2015,ChenWalshWheeler18}, we immediately obtain the following corollary, which states that all half-solitary waves are in fact supercritical solitary waves which therefore enjoy the symmetry and monotonicity properties established in~\cite{CraigSternberg88,Hur2008}. For a precise version see Corollary~\ref{c:main} below.
\begin{corollary}\label{c:informal}
  Any nontrivial bounded classical solution of \eqref{eqn:trav}--\eqref{a:unidirectional} is a supercritical and monotone solitary wave of elevation. Here symmetry and monotonicity mean that, after a translation in $x$, the free surface profile $\eta(x)$ is even, strictly positive, and satisfies $\eta'(x) < 0$ for $x > 0$.
\end{corollary}

Since it will be important for our proof, we conclude this subsection by introducing an invariant for steady waves called the \emph{flow force}~\cite{Benjamin84}, which is defined as
\begin{align}
  \label{eqn:FF}
  \FF = \int_0^{d+\eta} (P-P_\text{atm}+(u-c)^2)\, dy.
\end{align}
A direct calculation using \eqref{eqn:trav} shows that \eqref{eqn:FF} is a constant independent of $x$. The definition of $\FF$ is motivated by the divergence form of the momentum equations \eqref{eqn:u}--\eqref{eqn:v},
\begin{subequations}\label{eqn:div}
  \begin{align}
    \label{eqn:div1}
    \big(P+(u-c)^2\big)_x + \big((u-c)v\big)_y &= 0, \\
    \label{eqn:div2}
    \big((u-c)v\big)_x + \big(P+v^2+gy\big)_y &= 0.
  \end{align}
\end{subequations}

\subsection{Strategy of the proof}

The proof of Theorem~\ref{t:informal} centers on the flow force flux function $\flow(x,y)$ defined precisely in \eqref{def:ffff}. Roughly speaking, $\flow$ is the right hand side of \eqref{eqn:FF}, but with the upper limit of the integral replaced by $y$ and the limit as $x \to \infty$ subtracted off in an appropriate way. To our knowledge this particular function has not appeared before in the literature, and we believe that it will have further applications. In Proposition~\ref{prop:ffff} below, we prove that $\flow$ satisfies a linear elliptic equation to which the maximum principle applies. Since $\flow$ vanishes on the bed $y=0$ and is equal to $\eta^2 \ge 0$ on the surface, we deduce that it is strictly positive throughout the interior of the fluid. It is worth mentioning that, in the irrotational case, the function $\flow$ appears implicitly in Babenko's equation~\cite{Babenko87b} as the harmonic extension of $\eta^2$. We emphasize, however, that this description alone is not sufficient for our arguments.

Applying the Harnack inequality to $\flow \ge 0$, we see that a hypothetical subcritical half-solitary wave cannot decay at a super-exponential rate. On the other hand, it must have at least exponential decay thanks to the ``normal hyperbolicity''~\cite{Mielke1990} of the center manifold of small uniformly bounded solutions~\cite{Mielke88,GrovesWahlen08}, all of which are periodic. Arguing as in the supercritical case~\cite{Hur2008} then yields precise exponential asymptotics for the wave and hence for $\flow$ via its explicit definition. These asymptotics are not at all obvious from the characterization of $\flow$ as the solution to an elliptic boundary value problem, and in fact imply that $\flow$ changes sign near infinity, violating the maximum principle and leading to a contradiction.

The outline of the paper is as follows. In Section~\ref{s:prelim}, we give a precise statement of our main results and perform a standard change of variables which transforms~\eqref{eqn:trav} into an elliptic boundary value problem for a function $w(q,p)$ in an infinite strip. We also recall several well-known facts about the dispersion relation for the linearized problem. Section~\ref{s:proof} then contains the proof of our main result Theorem~\ref{t:informal} as well as Corollary~\ref{c:informal}.

\section{Statement of the problem} \label{s:prelim} \subsection{Notation} \label{s:notation}
For possibly unbounded domains $\Omega \subset \R^n$, we say that $u \in C^{k,\gamma}(\overline\Omega)$ if the usual $C^{k,\gamma}$ H\"older norm of $u$ is finite. 

\subsection{Reformulation} \label{s:formulation}
In this subsection we reformulate \eqref{eqn:trav}, under the important unidirectionality assumption \eqref{a:unidirectional}, as an elliptic equation \eqref{eqn:lin} for a function $w(q,p)$; the asymptotic condition \eqref{eqn:asym} becomes simply $w \to 0$ as $q \to +\infty$. Such transformations for water waves date back to the work of Dubreil--Jacotin~\cite{DubreilJacotin34} and are now completely standard, but we outline the main steps here so that the reader may more easily translate between the various notations.

For convenience, we will adopt the non-dimensional variables proposed by Keady and Norbury \cite{KeadyNorbury78}, which have length scale $(m^2/g)^{1/3}$ and velocity scale $(mg)^{1/3}$, where here the mass flux
\begin{align*}
  m = \int_0^{d+\eta(x)} (u(x,y)-c)\, dy
\end{align*}
is a constant independent of $x$ by incompressibility~\eqref{eqn:incomp} and the kinematic boundary conditions \eqref{eqn:kinbot}--\eqref{eqn:kintop}. In these units we now have $m=1$ and $g=1$.

Again by incompressibility \eqref{eqn:incomp}, there exists a stream function $\psi$ satisfying
\begin{align}
  \label{eqn:psi}
  \psi_x = -v, \qquad \psi_y = u-c. 
\end{align}
Moreover, the kinematic boundary conditions \eqref{eqn:kinbot}--\eqref{eqn:kintop} guarantee that $\psi$ is constant both on the bed $y=0$ and the free surface $y=d+\eta$. We normalize $\psi$ so that it is zero at the bed, in which case its value on the free surface is $m=1$. Taking the curl of \eqref{eqn:u}--\eqref{eqn:v}, we find that the vorticity
\begin{align}
  \label{eqn:omega}
  \omega = v_x - u_y 
\end{align}
and the stream function have parallel gradients, $\psi_x \omega_y - \psi_y \omega_x = 0$. Since our assumption \eqref{a:unidirectional} gives $\inf \psi_y > 0$, this then implies that $\omega$ is globally given as some nonlinear function of $\psi$. By an abuse of notation, we will write this relationship as
\begin{align*}
  \omega = \omega(\psi),
\end{align*}
and call $\omega(\psi)$ the \emph{vorticity function}. Substituting \eqref{eqn:psi} into \eqref{eqn:omega} now yields the semilinear 
elliptic equation 
\begin{equation*}
  \Delta\psi+\omega(\psi)=0. 
\end{equation*}

Next we eliminate the pressure $P$ using Bernoulli's law, which states that 
\begin{align}
  \label{eqn:bernoulli}
  P-P_\mathrm{atm} + \frac 12\abs{\nabla\psi}^2 + y  + \Omega(\psi) - \Omega(1) = R,
\end{align}
where here $R$ is the \emph{Bernoulli constant} and
\begin{align*}
  \Omega(\psi) = \int_0^\psi \omega(p)\,dp
\end{align*}
is a primitive of the vorticity function $\omega(\psi)$. Bernoulli's law can be easily verified by applying a gradient and using \eqref{eqn:u}--\eqref{eqn:v} to eliminate $P_x$ and $P_y$. Rewriting the dynamic boundary condition \eqref{eqn:dyn} using \eqref{eqn:bernoulli}, the traveling wave system \eqref{eqn:trav} becomes 
\begin{equation}
  \label{eqn:stream}
  \begin{alignedat}{2}
    \Delta\psi+\omega(\psi)&=0 &\qquad& \text{for } 0 < y < d + \eta,\\
    \tfrac 12\abs{\nabla\psi}^2 + y  &= R &\quad& \text{on }y=d+\eta,\\
    \psi  &= 1 &\quad& \text{on }y=d+\eta,\\
    \psi  &= 0 &\quad& \text{on }y=0.
  \end{alignedat}
\end{equation}
The asymptotic condition \eqref{eqn:asym} becomes
\begin{equation}
  \label{eqn:stream:asym}
  \eta(x) \to 0,
  \quad
  \psi(x,y) \to \Psi(y)
  \qquad 
  \text{as } x \to +\infty,
\end{equation}
where here
\begin{align*}
  \Psi(y) = \int_0^y (U(y)-c)\, dy
\end{align*}
must be related to $\omega(\psi),R,d$ through
\begin{align*}
  \Psi_{yy} + \omega(\Psi) = 0,
  \quad 
  \Psi(0)=0,
  \quad 
  \Psi(1)=1,
  \quad 
  \tfrac 12 \Psi_y^2(1) + d = R.
\end{align*}

While the reformulation \eqref{eqn:stream} of \eqref{eqn:trav} has many appealing features, it is still a free boundary problem in the sense the upper boundary $y=d+\eta$ of the fluid domain is itself an unknown. The unidirectionality assumption \eqref{a:unidirectional} allows us to flatten this boundary through an elegant change of coordinates~\cite{DubreilJacotin34}, at the cost of making the equations more nonlinear. Indeed, $\inf \psi_y > 0$ and so we can use
\begin{align*}
  q = x,
  \quad 
  p = \psi
\end{align*}
as new independent variables and
\begin{align*}
  y = h(q,p) 
\end{align*}
as the dependent variable. We call $h$ the \emph{height function}. The chain rule yields
\begin{align}
  \label{eqn:chainrule}
  \psi_x = -\frac{h_q}{h_p},
  \qquad 
  \psi_y = \frac 1{h_p},
\end{align}
so that in particular our assumption \eqref{a:unidirectional} implies
\begin{align*}
  \inf h_p > 0.
\end{align*}

Substituting \eqref{eqn:chainrule} into \eqref{eqn:stream}, one obtains the equivalent problem for the height function $h$,
\begin{subequations}\label{eqn:height}
  \begin{alignat}{2}
    \label{eqn:height:bulk}
    \biggl( \frac{1+h_q^2}{2h_p^2} + \Omega(p) \biggr)_p
    - \bigg( \frac{h_q}{h_p} \bigg)_q &= 0 &\qquad& \text{for } 0 < p < 1,\\
    \label{eqn:height:dynamic}
    \frac{1+h_q^2}{2h_p^2} + h  &= R
    && \text{on } p=1,\\
    \label{eqn:height:dirichlet}
    h &= 0 && \text{on } p=0,
  \end{alignat}
\end{subequations}
where the divergence form of \eqref{eqn:height:bulk} comes from \eqref{eqn:div2}; see~\cite{cs:discont}. Using \eqref{eqn:chainrule} and Bernoulli's law \eqref{eqn:bernoulli}, the flow force $\FF$ defined in \eqref{eqn:FF} is seen to be
\begin{align}
  \label{eqn:FFh}
  \FF = \int_0^1 \bigg(
  \frac{1-h_q^2}{2h_p^2} - h - \Omega(p)+\Omega(1) + R
  \bigg)h_p\, dp.
\end{align}
The asymptotic condition \eqref{eqn:stream:asym} is transformed into
\begin{equation}
  \label{eqn:hasym}
  h(q,p) \to H(p) \quad \text{as }q \to +\infty,
\end{equation}
where here $H$ is related to $\Psi$ via $y=H(\Psi(y))$ and must satisfy
\begin{align}
  \label{eqn:Hode}
  \Big( \frac 1{2H_p^2} + \Omega\Big)_p  = 0,
  \quad
  H(0)=0,
  \quad
  H(1)=d,
  \quad
  \frac 1{2H_p^2(1)} + d = R.
\end{align}
In particular, the Froude number $F$ is given by
\begin{equation*}
  \frac 1{F^2} = \int_0^1 H_p^3\, dp.
\end{equation*}

In light of \eqref{eqn:hasym}, it is natural to introduce the difference
\[
  w = h-H.
\]
Note that $w(q,1) = \eta(q)$. Substituting $h=H+w$ into \eqref{eqn:height} and isolating the linear terms, we obtain
\begin{equation}
  \label{eqn:lin}
  \begin{alignedat}{2}
    \biggl(\frac{w_p}{H_p^3}\biggr)_p + \biggl(\frac{w_q}{H_p}\biggr)_q  
    &=  N_1(w) &\qquad& \text{for } 0 < p < 1,\\ 
    -\frac{w_p}{H_p^3} +  w &= N_2(w) && \text{on } p=1, \\
    w &= 0 && \text{on } p=0, 
  \end{alignedat}
\end{equation}
where the asymptotic condition \eqref{eqn:hasym} becomes simply
\begin{equation}
\label{eqn:wasym}
w(q,p) \to 0 \quad \text{as } q\to +\infty.
\end{equation}
The nonlinear terms on the right-hand side of \eqref{eqn:lin} are given by
\begin{align*}
   N_1(w) &=  
  \bigg( \frac{H_p^3 w_q^2 + (2w_p+3H_p)w_p^2}{2H_p^3(w_p+H_p)^2}\bigg)_p
  + \bigg( \frac{w_pw_q}{H_p(w_p+H_p)} \bigg)_q,\\
  N_2(w) &= 
 -\frac{H_p^3 w_q^2 + (2w_p+3H_p)w_p^2}{2H_p^3(w_p+H_p)^2}.
\end{align*}

\subsection{Dispersion relation}\label{s:dispersion}

The following linearized version of \eqref{eqn:lin}, in which $N_1$ and $N_2$ have been set equal to zero, plays a central role in our arguments:
\begin{align}
\label{eqn:linhom}
\begin{alignedat}{2}
\biggl(\frac{v_p}{H_p^3}\biggr)_p + \biggl(\frac{v_q}{H_p}\biggr)_q & = 0 &\qquad& \text{for } 0<p<1, \\
-\frac{v_p}{H_p^3} +  v & = 0 && \text{on } p=1,\\
v& = 0 && \text{on } p=0.
\end{alignedat}
\end{align}
This is a first-order approximation for small-amplitude water waves, and the leading-order approximation for small-amplitude water waves which are periodic in $q$.

Separating variables as usual, we find that bounded solutions of \eqref{eqn:linhom} are linear combinations of the functions
\[
\cos(\abs{-\lambda_j}^{1/2}q) \varphi_j(p), \quad \sin(\abs{-\lambda_j}^{1/2}q) \varphi_j(p)
\]
where $\lambda_j$ and $\varphi_j$ are the non-positive eigenvalues and corresponding eigenfunctions of the Sturm--Liouville problem
\begin{equation}
  \label{eqn:SL}
  \begin{alignedat}{2}
    - \left( \frac{\varphi_p}{H_p^3} \right)_p & = \lambda \frac{\varphi}{H_p} 
    &\qquad& 0 < p < 1,\\ 
    -\frac{\varphi_p}{H_p^3} + \varphi &= 0 && p=1,\\ 
    \varphi & = 0 && p=0. 
  \end{alignedat}
\end{equation}
It is well known that Sturm--Liouville problems such as \eqref{eqn:SL} have a countable set of simple eigenvalues $\lambda_0 < \lambda_1 < \cdots < \lambda_j < \cdots$ accumulating at infinity, and the corresponding eigenfunctions can be rescaled to form an orthonormal basis for $L^2(0,1;H_p^{-1})$, the space of $L^2$ functions on $(0,1)$ with measure $H_p^{-1}\,dp$. Classical oscillation theory asserts that $\varphi_j$ has precisely $j$ zeros on the interval $(0,1)$~\cite[theorem~5.11]{Teschl12}. Moreover, a calculation shows that $\lambda_0 < 0$ if and only if $F<1$,  and in this case $\lambda_1 > 0$; see for instance \cite[theorem~5.17]{Teschl12}. Throughout this paper we are assuming $F < 1$ so that the wave is subcritical, and therefore we write 
\begin{align*}
\lambda_0 = -\tau^2_0 < 0, 
\qquad \lambda_j = \tau_j^2 > 0 \quad \text{for }j=1,2,\ldots.
\end{align*}

\subsection{Precise statement of the main result}\label{s:precise}
A more precise version of our main result Theorem~\ref{t:informal}, as well as Corollary~\ref{c:informal}, are as follows. We assume that the vorticity function $\omega \in C^\gamma$, or equivalently that $H \in C^{2,\gamma}$.
\begin{theorem} \label{t:main} 
  Let $w \in C^{2,\gamma}(\bar S)$ be a solution of \eqref{eqn:lin}--\eqref{eqn:wasym} with subcritical Froude number $F < 1$. Then $w \equiv 0$.
\end{theorem}
Using Theorem~\ref{t:main}, we can now prove a precise version of Corollary~\ref{c:informal}.
\begin{corollary}\label{c:main}
  Let $w \in C^{2,\gamma}(\bar S)$ solve \eqref{eqn:lin}--\eqref{eqn:wasym}, and assume that $w \not \equiv 0$. Then $F > 1$, and $w$ is a monotone solitary wave of elevation in that, after a translation, $w$ is even in $q$ and satisfies $w_q < 0$ for $q > 0$ and $0 < p \le 1$.
  \begin{proof}
    Let $w$ be as in the statement of the corollary. By Theorem~\ref{t:main}, $w$ cannot be subcritical, and from~\cite{Kozlov2017a} (also see~\cite{Wheeler15b}) we know that $w$ cannot be critical either. Thus $w$ is supercritical, and hence $w > 0$ for $0 < p \le 1$ by the maximum principle argument in~\cite{Wheeler13,Kozlov2015}. The hypotheses of the moving planes argument in~\cite{ChenWalshWheeler18} are therefore satisfied, which gives the result. Here we cite~\cite{ChenWalshWheeler18} because we assume decay only as $q \to +\infty$; the corresponding result with decay in both directions is originally due to Hur~\cite{Hur2008} with vorticity and \cite{CraigSternberg88} in the irrotational case.
  \end{proof}
\end{corollary}

\section{Proof of the main result}\label{s:proof}

The proof of Theorem~\ref{t:main} centers on the properties of the (relative) flow force flux function $\flow$, which we believe is introduced in the present paper for the first time. Its definition in \eqref{def:ffff} below is motivated by the following calculation.
Suppose $h \in C^{2,\gamma}(\bar S)$ is a solution of \eqref{eqn:height} but not necessarily \eqref{eqn:hasym}, and let $H(p)$ be a solution of \eqref{eqn:height} with the same Bernoulli constant $R$, i.e.~a solution of \eqref{eqn:Hode}. We seek a simple formula for the difference $\FF-\FF_+$, where $\FF$ is the flow force constant \eqref{eqn:FFh} for $h$ and $\FF_+$ is the constant for $H$. Replacing $h$ with $H$ in \eqref{eqn:FFh} we see that 
\begin{align}
  \label{eqn:FFplus}
  \FF_+ = \int_0^1 \bigg( \frac 1{2H_p^2(p)} - H - \Omega(p) + \Omega(1) + R \bigg)H_p(p)\, dp,
\end{align}
and so setting $w=h-H$ yields
\begin{align*}
  \FF - \FF_+ &= \int_0^1 I(q,p)\, dp,
\end{align*}
where the integrand is given by
\begin{align*}
  I &= \bigg(
\frac{1-w_q^2}{2h_p^2} - h - \Omega+\Omega(1) + R
\bigg)h_p
 -  \bigg( \frac 1{2H_p^2} - H - \Omega + \Omega(1) + R \bigg)H_p \\
&= \bigg( \frac{1-w_q^2}{2h_p^2}-\frac 1{2H_p^2} - w \bigg)H_p
+  \bigg(
\frac{1-w_q^2}{2h_p^2} - h - \Omega+\Omega(1) + R
\bigg)w_p \\
  & = \frac 1{2h_p} - \frac{w_q^2}{2h_p} - \frac 1{2H_p} - hh_p + Hh_p - Hw_p
  + (-\Omega + \Omega(1) + R)w_p.
\end{align*}
Substituting the identity
\[
  -\Omega(p) + \Omega(1)+R = \frac{1}{2H_p^2(p)} + H(1),
\]
which follows directly from \eqref{eqn:Hode}, we split $I$ into three terms,
\begin{equation}
  \label{eqn:I123}
  \begin{aligned}
    I &=  -\frac{w_q^2}{2h_p} + \frac 12 \Big( \frac 1{h_p} - \frac 1{H_p} + \frac{w_p}{H_p^2} \Big)
    + (-hh_p + Hh_p - Hw_p + H(1)w_p)\\
    &=: I_1 + I_2 + I_3.
  \end{aligned}
\end{equation}
Putting $I_2$ over a common denomiator yields
\begin{align*}
  I_2 = \frac{H_p^2 - H_p h_p + h_p w_p}{h_p H_p^2} = \frac{w_p^2}{h_pH_p^2},
\end{align*}
while $I_3$ simplifies to 
\begin{align*}
  I_3 = -Hw_p - wH_p -ww_p + H(1)w_p
  = \Big( {-Hw} + H(1)w - \tfrac 12 w^2 \Big)_p.
\end{align*}
Substituting back into \eqref{eqn:I123} and integrating the total derivative, we find
\begin{equation}\label{eqn:rhsme}
  2(\FF-\FF_+)= -w^2(q,1) + \int_0^1 \bigg( \frac{w_p^2}{h_p H_p^2} - \frac{w_q^2}{h_p} \bigg) dp.
\end{equation}
We define the (relative) flow force flux function $\flow(q,p)$ to be the integral in \eqref{eqn:rhsme} but with the upper limit replaced by $p$,
\begin{equation} \label{def:ffff}
  \flow(q,p) = \int_0^p \bigg( \frac{w_p^2(q,p')}{h_p(q,p') H_p^2(p')} - \frac{w_q^2(q,p')}{h_p(q,p')} \bigg) dp'.
\end{equation}

A surprising new fact about $\flow$ is that it solves a homogeneous elliptic equation as stated in the following
\begin{proposition}\label{prop:ffff} 
  Let $h, H \in C^{2,\gamma}(\bar S)$ be as above and define the flow force flux function $\flow$ by \eqref{def:ffff}. Then there exist coefficients $b_1,b_2 \in L^{\infty}(S)$ such that $\flow \in C^{2,\gamma}(\bar S)$ satisfies the linear equation
  \begin{align}
    \label{eqn:elliptic}
    \frac{1+h_q^2}{h_p^2}
    \flow_{pp}
    - 2\frac{h_q}{h_p} \flow_{pq}  + \flow_{qq}  + 
    b_1 \flow_q + b_2
    \flow_p 
    = 0 \quad \textup{in } S,
  \end{align}
  together with the boundary conditions 
  \begin{align}
    \label{eqn:bc}
    \flow = 0 \quad\textup{on } p=0,
    \qquad 
    \flow = w^2 + 2(\FF-\FF_+) \quad \textup{on } p=1. 
  \end{align}
  In the irrotational case with no vorticity, $b_1,b_2 \equiv 0$ and \eqref{eqn:elliptic} is equivalent to Laplace's equation $(\partial_x^2 + \partial_y^2)\flow = 0$ in the original variables.
\end{proposition}

Note that, when the asymptotic condition \eqref{eqn:wasym} holds, $\FF = \FF_+$ so that $\flow$ is positive in $S$ by the maximum principle applied to \eqref{eqn:elliptic}--\eqref{eqn:bc}.

\begin{proof} 
  The boundary condition at $p=0$ is immediate, while the condition at $p=1$ follows directly from \eqref{eqn:rhsme}, and so it remains to show that $\flow \in C^{2,\gamma}(\bar S)$ solves \eqref{eqn:elliptic}.
 
  First, let us compute $\flow_q$. Differentiating \eqref{def:ffff} with respect to $p$ and then with respect to $q$, we find
\[
  \begin{split}
    \flow_{pq} &= \frac{2w_p}{h_pH_p^2}w_{pq}-\frac{w_p^2}{h_p^2H_p^2}w_{pq}-2w_q \bigg(\frac{w_q}{h_p}\bigg)_q-\frac{w_q^2}{h_p^2} w_{pq} \\
    &= - \bigg(\frac{1+w_q^2}{h_p^2}-\frac{1}{H_p^2}\bigg)w_{qp}-w_q\bigg(\frac{1+w_q^2}{h_p^2}-\frac{1}{H_p^2}\bigg)_p \\
    & = -\bigg[ w_q \bigg(\frac{1+w_q^2}{h_p^2}-\frac{1}{H_p^2}\bigg) \bigg]_p.
  \end{split}
\]
Here we used the identity
\[
\bigg(\frac{1+w_q^2}{2h_p^2}-\frac{1}{2H_p^2}\bigg)_p =  \bigg(\frac{w_q}{h_p}\bigg)_q,
\]	
which is the difference of \eqref{eqn:height:bulk} and the first equation in \eqref{eqn:Hode}. Taking into account that $w_q$ is zero along the bottom, we conclude 
\begin{align}
    \label{eqn:flowq}
    \flow_q = -w_q \bigg(\frac{1+w_q^2}{h_p^2}-\frac 1{H_p^2}\bigg).
\end{align}
Now the definition of $\flow$ and \eqref{eqn:flowq} show $\flow \in C^{2,\gamma}(\bar S)$. 
 
A direct computation gives the following formulas for the second-order derivatives of $\flow$:
\begin{align*}
\flow_{qq} & = -\frac{\left(-2 H_p w_p+3 H_p^2 w_q^2-w_p^2\right)}{H_p^2 h_p^2}w_{qq} + \frac{2w_q(1+w_q^2)}{h_p^3} w_{qp}, \\
\flow_{qp} & = - \frac{2 w_q}{h_p} w_{qq} + \frac{2H_pw_p + w_p^2 + H_p^2w_q^2}{H_p^2 h_p^2}w_{qp}, \\
\flow_{pp} & = - \frac{2 w_q}{h_p} w_{qp} + \frac{2H_pw_p + w_p^2 + H_p^2w_q^2}{H_p^2 h_p^2}w_{pp} + \frac{H_{pp}(-3H_pw_p^2-2w_p^3+H_p^3w_q^2)}{H_p^3 h_p^2}.
\end{align*}
We use these formulas to compute
\begin{align}
  \label{eqn:tocompute}
  \begin{aligned}
    \frac{1+h_q^2}{h_p^2}
    \flow_{pp}- 2\frac{h_q}{h_p} \flow_{pq} + \flow_{qq} & = \frac{2H_pw_p + w_p^2 + H_p^2w_q^2}{H_p^2 h_p^2} \left(\frac{1+h_q^2}{h_p^2}
    w_{pp}- 2\frac{h_q}{h_p} w_{pq} + w_{qq}\right) \\
    & \qquad + \frac{1+w_q^2}{h_p^2} \, \frac{H_{pp}(-3H_pw_p^2-2w_p^3+H_p^3w_q^2)}{H_p^3 h_p^2}.
  \end{aligned}
\end{align}
Since $w$ solves 
\[
\frac{1 + h_q^2}{h_p^2} w_{pp} - 2 \frac{h_q}{h_p} w_{qp} + w_{qq}  = \frac{H_{pp} h_p}{H_p^3} \left(1 - \frac{H_p^3(1+w_q^2)}{h_p^3}\right),
\]
we can rewrite \eqref{eqn:tocompute} as
\begin{align}
    \label{eqn:inhom}
    \frac{1+h_q^2}{h_p^2}
          \flow_{pp}
          - 2\frac{h_q}{h_p} \flow_{pq} + \flow_{qq}  - 
    \frac{H_{pp}(w_p-H_p)}{H_p^3 h_p} \flow_p
    = \frac{4H_{pp}w_p^2}{H_p^4 h_p^2}.
\end{align}
  When the flow is irrotational so that $H_{pp} = 0$, both the source term on the right hand side  and the coefficient of $\flow_p$ vanish. Indeed, in this very special case \eqref{eqn:inhom} is precisely the statement that $\flow$ is a harmonic function of the physical variables $x$ and $y$. In general, however, the source term does not have a definite sign, and so \eqref{eqn:inhom} is insufficient for our purposes.
  
  To circumvent this deficiency, we use \eqref{def:ffff} and \eqref{eqn:flowq} to ``eliminate'' the factor of $w_p^2$ on the right hand side of \eqref{eqn:inhom} in favor of $\flow_p$ and $\flow_q$. Linear combinations of our formulas for $\flow_p$ and $\flow_q$ give the two identities
  \begin{align}
    \label{eqn:killwp}
    w_p^2 &= H_p^2 w_q^2 + H_p^2 h_p \flow_p,\\
    \label{eqn:wpwq}
    w_p w_q &= \frac{H_p h_p}{2} (h_p \flow_q - h_q \flow_p).
  \end{align}
  Combining \eqref{eqn:killwp} and \eqref{eqn:wpwq} yields a quadratic equation for $w_p^2$:
  \begin{equation*}
    w_p^4 = \frac{H_p^4 h_p^2}{4} (h_q \flow_p  - h_p \flow_q)^2
    + H_p^2 h_p w_p^2 \flow_p.
  \end{equation*}
  Since $w_p^2 \ge 0$, this has the unique solution
  \begin{equation}
    \begin{split}
      \label{eqn:wp2}
      w_p^2 &= \frac{H_p^2 h_p}{2} 
      \Big(\flow_p
      + \sqrt{\flow_p^2 + (\flow_p h_q - \flow_q h_p)^2}\Big)\\
      &= \frac{H_p^2 h_p}{2} (B_1 \flow_q + B_2 \flow_p),
    \end{split}
  \end{equation}
  where the coefficients $B_1,B_2$ are well-defined $L^\infty$ functions on $S$:
  \begin{equation*}
    \begin{aligned}
      B_1 &= 
      \frac{\sign(\flow_p h_q - \flow_q h_p)\sqrt{\flow_p^2 + (\flow_p h_q - \flow_q h_p)^2}}{\abs{\flow_p} + \abs{\flow_p h_q - \flow_q h_p}},\\
      B_2 &= 1 + 
      \frac{\sign(\flow_p)\sqrt{\flow_p^2 + (\flow_p h_q - \flow_q h_p)^2}}{\abs{\flow_p} + \abs{\flow_p h_q - \flow_q h_p}}.
    \end{aligned}
  \end{equation*}
  Substituting \eqref{eqn:wp2} into \eqref{eqn:inhom} yields the \emph{homogeneous} linear equation
  \begin{equation*}
    \frac{1+h_q^2}{h_p^2}
    \flow_{pp}
    - 2\frac{h_q}{h_p} \flow_{pq}+ \flow_{qq}- 
    \bigg(\frac{H_{pp}(w_p-H_p)}{H_p^3 h_p} + 2\frac{H_{pp}}{H_p^2 h_p}B_2\bigg) \flow_p - 2\frac{H_{pp}}{H_p^2 h_p} B_1
    \flow_q 
    = 0,
  \end{equation*}
  which is of the desired form \eqref{eqn:elliptic}.
\end{proof}

We note that, for irrotational waves, the recent paper \cite{Basu2019} considers a ``flow force function'' obtained by replacing the upper limit of integration in \eqref{eqn:FFh} by $p$. 

\begin{proposition}\label{p:exact}
  Any solution $w \in C^{2,\gamma}(\bar S)$ to \eqref{eqn:lin}--\eqref{eqn:wasym} which is not identically zero satisfies
  \begin{equation}\label{eqn:thmasymp}
          w(q,p) = a \varphi(p)e^{-\tau q} + e^{-\tau'q}z(q,p),
  \end{equation}
  where $a \neq 0$, $z \in C^{2,\gamma}(\bar{S})$, $\tau'>\tau>0$, and $\lambda=\tau^2$ is a positive eigenvalue of the Sturm-Liouville problem \eqref{eqn:SL} with eigenfunction $\varphi$.
\end{proposition}
We emphasise that $\lambda = \tau^2$ is a positive eigenvalue of \eqref{eqn:SL} and not the unique negative eigenvalue $\lambda = -\tau_0^2$. Thus the corresponding eigenfunction $\varphi(p)$ in \eqref{eqn:thmasymp} must change sign.
 
\begin{proof} 
  With only the weak assumption \eqref{eqn:wasym} about the decay of $w$, exact asymptotics such as \eqref{eqn:thmasymp} can be difficult to obtain. Thankfully, in our case a result of Mielke~\cite{Mielke1990} implies that $w$ decays exponentially. Indeed, applying theorem~3.1b in \cite{Mielke1990} one finds that $w$ converges at some exponential rate to a small-amplitude solution $v$ of \eqref{eqn:lin}. By the center manifold construction of Groves and Wahl\'en~\cite{GrovesWahlen08}, this small solution $v$ must be periodic, at which point the asymptotic condition \eqref{eqn:wasym} for $w$ forces $v\equiv 0$. Thus $w$ decays to zero at an exponential rate.
 
  Having established exponential decay for $w$, the exact asymptotics \eqref{eqn:thmasymp} become standard and well known (see \cite{Hur2008} or part III in \cite{Kozlov1999}). One only has to check that $\tau$ can be chosen so that the coefficient $a$ in \eqref{eqn:thmasymp} is nonzero, i.e.~check that $w$ does not decay at a hyper-exponential rate. But this is not possible in view of Proposition~\ref{prop:ffff}. Indeed, the asymptotic condition \eqref{eqn:wasym} implies that $\FF = \FF_+$ and so $\flow = w^2 \ge 0$ on $p=1$, while $\flow = 0$ for $p=0$. Proposition~\ref{prop:ffff} and the maximum principle for unbounded domains \cite{Vitolo2007} therefore imply that $\flow > 0$ everywhere in $S$. Hyper-exponential decay of $\flow$ is thus forbidden by the Harnack inequality. Since hyper-exponential decay for $w$ would imply hyper-exponential decay for $\flow$ by its definition in \eqref{def:ffff}, the proof is complete.
\end{proof}

Combining Propositions~\ref{prop:ffff} and \ref{p:exact}, it is easy to finalize the proof of Theorem \ref{t:main}. 
\begin{proof}[Proof of Theorem~\ref{t:main}]
  Substituting the asymptotics \eqref{eqn:thmasymp} into \eqref{def:ffff}, we find that
  \begin{align*}
    \flow(q,p) 
    &= \int_0^p 
    \bigg[ 
    \frac{a^2}2\bigg(\frac{(e^{-\tau q} \varphi)_p^2}{H_p^3} - \frac{(e^{-\tau q}
        \varphi)_q^2}{H_p^2}\bigg) + O(e^{-(\tau+\tau') q})
    \bigg] \, dp'\\
    &= \frac{a^2 e^{-2\tau q}}2 \int_0^p 
    \bigg(\frac{\varphi_p^2}{H_p^3} - \tau^2 \frac{\varphi^2}{H_p}\bigg) 
    \, dp'
    + O(e^{-(\tau+\tau') q}).
  \end{align*}
  Since $\varphi$ solves \eqref{eqn:SL} with $\lambda=\tau^2$, we can calculate
  \begin{align*}
    \int_0^p 
    \bigg(\frac{\varphi_p^2}{H_p^3} - \tau^2 \frac{\varphi^2}{H_p}\bigg) 
    \, dp'
    = \frac{\varphi(p)\varphi_p(p)}{H_p^3(p)} ,
  \end{align*}
  and so the above expansion simplifies to
  \begin{align}
    \label{eqn:contradiction}
    \flow(q,p)
    &= \frac{a^2 \varphi(p) \varphi_p(p)}{2 H_p^3(p)}
    e^{-2\tau q}
    + O(e^{-(\tau+\tau') q}).
  \end{align}
  We know that $\varphi^2$ has a root in $(0,1)$ and also that $\varphi^2(0)=0$. Thus, its derivative $2\varphi\varphi_p$ changes sign on $(0,1)$. By \eqref{eqn:contradiction}, the same must be true for $\flow$ when $q$ is sufficiently large. On the other hand $\flow > 0$ everywhere in $S$ by the maximum principle, which leads to a contradiction.
\end{proof}

\subsection*{Acknowledgments} 
Large parts of this research were carried out while E.L. and M.H.W. were at Mathematisches Forschungsinstitut Oberwolfach for a Research in Pairs program. V.K. was supported by the Swedish Research
Council (VR), 2017-03837.
\bibliographystyle{alpha}
\bibliography{arxiv}
\end{document}